\documentclass[12pt]{amsart}

\usepackage{amsmath, amssymb, amsthm, latexsym,nicefrac,setspace}
\usepackage[pagebackref,colorlinks,linkcolor=red,citecolor=blue,urlcolor=blue,hypertexnames=true]{hyperref}
\usepackage{amsrefs}
\usepackage{marginnote}

\theoremstyle{plain}
\newtheorem{theorem}{Theorem}[section]

\newtheorem{lemma}[theorem]{Lemma}
\newtheorem{corollary}[theorem]{Corollary}
\newtheorem{proposition}[theorem]{Proposition}
\newtheorem{definition}[theorem]{Definition}
\newtheorem{conjecture}[theorem]{Conjecture}
\newtheorem{example}[theorem]{Example}
\newtheorem*{theorem*}{Theorem}
\theoremstyle{remark}
\newtheorem{remark}[theorem]{Remark}

\input{xy}
\xyoption{all}

\def\Z{\mathbb Z}

\def\C{\mathbb{C}}
\def\R{\mathbb{R}}
\def\bR{{\bf R}}
\def\id{\mathrm{id}}
\def\N{\mathbb{N}}

\def\bC{{\bf C}}

\title[Real Closed Separation Theorems]{Real Closed Separation Theorems and Applications to Group Algebras}

\author{Tim Netzer}
\author{Andreas Thom}

\begin{document}

\onehalfspace

\begin{abstract} In this paper we prove a strong Hahn-Banach theorem: separation  of disjoint convex sets by linear forms is possible without any further conditions, if the target field $\R$ is replaced by a more general real closed extension field. From this we deduce a general Positivstellensatz for $*$-algebras, involving representations over real closed fields. We investigate the class of group algebras in more detail. We show that the cone of sums of squares in the augmentation ideal has an interior point if and only if the first cohomology vanishes. For groups with Kazhdan's property (T) the result can be strengthened to interior points in the $\ell^1$-metric. We finally reprove some  strong Positivstellens\"atze by Helton and  Schm\"udgen, using our separation method. \end{abstract}

\maketitle

\tableofcontents

\section{Introduction}

In this article we combine techniques from real algebraic geometry, convex geometry, and the unitary representation theory of discrete groups to address various problems which arise in the emerging field of non-commutative real algebraic geometry \cite{MR2500470}. Classical results -- like Artin's solution of Hilbert's 17$^{\rm th}$ problem -- strive for a characterization of natural notions of positivity in terms of algebraic certificates. For example, Artin proved that every polynomial in $n$ variables, which is positive at every point on $\R^n$, must be a sum of squares of rational functions. Much later, Schm\"udgen \cite{schm2} proved that a strictly positive polynomial on a compact semi-algebraic set must be a sum of squares of polynomials plus defining inequalities. More recently, similar questions were asked in a non-commutative context. Typically, the setup involves a $\ast$-algebra $A$ and a family of representations $\mathcal F$. The question is now: Is every self-adjoint element of $A$, which is positive (semi-)definite in every  represention in $\mathcal F$ necessarily of the form $\sum_i a_i^*a_i$ for some $a_i \in A$? 
It turned out -- similar to the more classical commutative case -- that the cone $\Sigma^2 A = \{ \sum_i a_i^*a_i \mid a_i \in A \} \subset A$ is an interesting object of study in itself. Natural questions are:  Is $\Sigma^2 (A) \cap (- \Sigma^2 A) = \{0\}$?
 Is $\Sigma^2 A$ closed in a natural topology? Does it contain interior points?

The question for interior points of cones has the following motivation. If a cone $C$ has an (algebraic) interior point $q$, then for every point $a\in C^{\vee\vee}$ from the double dual, one has $a+\epsilon q\in C,$ for all $\epsilon >0 $ (see for example Proposition 1.3 in \cite{CMN} for a proof of  this well-known fact). Using the standard GNS-construction, this yields the following Positivstellensatz for unital $*$-algebras: \begin{theorem*} Assume  that $q$ is an interior point of the cone $\Sigma^2 A$. If $a=a^*\in A$ is positive semidefinite in each $*$-representation of $A$, then $a+\epsilon q\in \Sigma^2 A$ for all $\epsilon >0.$
\end{theorem*}

Our first main result is a different Positivstellensatz (Theorem \ref{pos}): we prove that each element from a real reduced unital $*$-algebra, which is positive in every generalized representation, is necessarily in $\Sigma^2 A$.  The notion of a \textit{generalized representation} involves an extension of the standard real and complex numbers to more general real- and algebraically closed fields.

A natural and vast class of examples of $*$-algebras is given by complex group algebras $\C[\Gamma]$ of discrete countable groups. We study the cones $\Sigma^2 \C[\Gamma]$ and $\Sigma^2 \omega(\Gamma)$ in more detail, where $\omega(\Gamma) \subset \C[\Gamma]$ denotes the augmentation ideal, see Section \ref{aug}. The situation for $\omega(\Gamma)$ is much more complicated, as the study is closely related to questions about first cohomology with unitary coefficients.We prove that $\Sigma^2 \omega(\Gamma)$ has a interior point if and only if $H_1(\Gamma,\C)=0$. The cone $\Sigma^2 \omega(\Gamma)$ has interior point in the $\ell^1$-metric if $\Gamma$ has Kazhdan's property (T), and the converse holds if $H_2(\Gamma,\C)=0$ (see Section \ref{sec:kaz}).
 In Section \ref{free}, we analyze the situation for free groups more closely and reprove Theorems of Schm\"udgen and Helton. 
\vspace{0.2cm}

Along the way we prove some new and powerful separation theorems in Sections \ref{sec:sep} and \ref{csi}. The Hahn-Banach separation theorems for convex sets only apply if additional conditions on the involved sets are imposed; sets have to be closed or have to have non-empty interiour etc. We can remove {\bf all} additional assumptions on the expense of enlarging the target $\R$ to some real closed extension of the real numbers, see Theorem \ref{separationtheorem}.

\section{A Real Closed Separation Theorem for Convex Sets}\label{sec:sep}

 Throughout  we will work with various real closed fields $\bR$ and {\bf always} assume that $\R \subset \bR$. The following is a first general separation theorem for convex cones.
 
 \begin{theorem}\label{separationtheorem}Let $V$ be an $\R$-vector space,  $C\subset V$ a convex cone and $x \not \in C$. Then there exists a real closed field $\bR$ containing $\R$, and an $\R$-linear functional $\varphi \colon V \to R$, such that 
$$\varphi(x)< 0 \quad \mbox{and} \quad \varphi(y) \geq  0 \quad \forall y \in C.$$ We can even ensure $\varphi(y)>0$ for all $y\in C\setminus\left( C\cap -C\right).$ Also, $\bR$ does only depend on $V$, not on $x$ or $C$.
\end{theorem}
\begin{proof} 
Let us first assume that $V$ is finite-dimensional. We construct a complete flag of subspaces $V=H_1\supset H_2 \supset \cdots \supset H_n=C\cap -C$, starting with $V=H_1$,  in the following way. By the standard separation theorem for convex sets (see for Example Theorem 2.9 in \cite{barvinok}) we choose a non-trivial  $\R$-linear functional $\varphi_i\colon H_i \rightarrow \R$ such that  $\varphi(y)\geq 0$ for all $y\in C\cap H_i$ and $\varphi_i(x)\leq 0$ (if $x\in H_i$). We then define $H_{i+1}:=H_i\cap\{\varphi_i=0\}$ and iterate the process. We finally extend each   $\varphi_i$ in any way to $V$.
Now let $\bR$ be a real closed extension field of $\R$. Choose positive elements $$1=\epsilon_1 > \epsilon_2> \cdots >\epsilon_{n-1}>0$$ from $\bR$ such that $k\cdot \epsilon_i < \epsilon_{i-1}$ for all $k\in\R.$ Then define $$\varphi:= \epsilon_1\varphi_1 + \cdots +\epsilon_{n-1}\varphi_{n-1}.$$
One checks that $\varphi$ has the desired properties. This proves the claim in case of finite dimension.
 
In the general case, consider the set $\mathcal{S}$ of all finite dimensional subspaces $H$ of $V$. For each $H\in \mathcal{S}$ choose $\varphi_H\colon H\rightarrow \bR$, separating $x$ from $C\cap H$ as desired (if $x\in H$). Extend $\varphi_H$ in any way to $V$.  Now let $\omega$ be an ultrafilter on $\mathcal{S}$, containing the sets $\{H\in\mathcal{S}\mid y\in H\}$, for all $y\in V$. Consider the linear functional $\varphi\colon V\rightarrow \bR^{\omega}$, $\varphi(v)=\left(\varphi_H(v)\right)_{H\in\mathcal{S}}$. Here $\bR^\omega$ denotes the ultrapower of $\bR$ with respect to $\omega$. One checks that $\varphi$ separates $x$ as desired, by the Theorem of \L os (see for example Theorem 2.2.9 in \cite{pd}).\end{proof}

\begin{remark}
In the usual way, one can  now also deduce that any two convex disjoint sets in a vector space can be separated as above with a real-closed valued affine  functional.
\end{remark}

It turns out that we can also extend functionals quite often, if we allow for an extension of the real closed field.
\begin{theorem}\label{extend} Let $V$ be an $\R$-vector space, $C\subseteq V$ a convex cone and $H\subseteq V$ a subspace. Assume $(C+H)\cap-(C+H)=H$. Then for any real closed extension  field $\bR$ of $\R$ and any $\R$-linear functional $\varphi\colon H\rightarrow \bR$ with $\varphi\geq 0$ on $C\cap H$   there is some real closed extension field $\bR'$ of $\bR$ and an $\R$-linear functional $\overline\varphi\colon V\rightarrow \bR'$ with $\overline\varphi\geq 0$ on $C$ and $\overline\varphi=\varphi$ on $H$. We can even ensure $\overline\varphi(y)>0$ for all $y\in C\setminus H$.
\end{theorem}
\begin{proof} We apply Theorem \ref{separationtheorem} to the convex cone $C+H$ in $V$, and obtain a  real closed field $\widetilde\bR$ and a nontrivial  $\R$-linear functional $\psi\colon V\rightarrow\widetilde\bR$ with $\psi=0$ on $H$ and $\psi(y)>0$ for $a\in C\setminus H.$ By amalgamation of real closed fields we can assume without loss of generality that $\bR=\widetilde\bR$. Let finally $\bR'$ be a real closed extension field of $\bR$ that contains an element $\delta>\bR$. Extend $\varphi$ to an $\bR$-valued funtional on $V$ and set $$\overline\varphi:=\varphi + \delta\cdot \psi.$$  It is clear that $\overline\varphi $ coincides with $\varphi$ on $H$ and also that $\overline\varphi(y)> 0$ for all $y\in C\setminus H$.
\end{proof}

We will improve upon the separation results in the case of certain $*$-algebras in the next section.

\section{Completely Positive Separation}\label{csi}

Throughout this section let $A$ be a $\C$-algebra with involution $*$,  not necessarily unital. We consider the cone of sums of hermitian squares $$\Sigma^2 A=\left\{ \sum_{i=1}^n a_i^*a_i\mid n\in\N, a_i\in A\right\},$$ contained in the real vector subspace of hermitian elements $$A^h=\left\{ a\in A\mid a^*=a\right\}.$$ If $b\in A^h\setminus\Sigma^2A$, we find an $\R$-linear functional  $\varphi\colon A^h\rightarrow \bR,$ into some real closed extension field $\bR$ of $\R$, such that $\varphi(b)<0$, $\varphi(a^*a)\geq 0$ for all $a\in A$, by Theorem \ref{separationtheorem} from the last section.  We can extend $\varphi$ uniquely to a $\C$-linear functional $\varphi\colon A\rightarrow \bR[i]$ fulfilling $\varphi(a^*)=\overline{\varphi(a)}$. We will denote the algebraically closed field $\bR[i]$ by $\bC$ from now on.

The condition $\varphi(a^*a)\geq 0$ for all $a\in A$  is called {\it positivity} of $\varphi$. We would now  like positive and real-closed valued functionals to fulfill the Cauchy-Schwarz inequality $$\vert\varphi(a^*b)\vert^2\leq \varphi(a^*a)\varphi(b^*b)$$  for all $a,b\in A$. This is however not true in general, as shows the next example.

\begin{example}\label{notcp} Let $A=\C[t]$ be the univariate polynomial ring, where $*$ is coefficient-wise conjugation. The cone $\Sigma^2 A$ equals the cone of nonnegative real polynomials. Consider the functional $$\varphi\colon\R[t]\rightarrow \bR, \quad p\mapsto p(0)+\epsilon p''(0),$$ where $\epsilon\in\bR $ is positive and  infinitesimal with respect to $\R$. One checks that $\varphi$ is positive, but for $a=1+t^2$ and $b=1$ we have $$\vert \varphi(a^*b)\vert^2= 1+4\epsilon + 4\epsilon^2> 1+4\epsilon = \varphi(a^*a)\varphi(b^*b).$$ 
\end{example}

\begin{example}\label{notcp2}
The last example can be modified to even fulfill $\varphi(a^*a)>0$ if $a\neq 0$. Indeed let $1=\epsilon_0 >\epsilon_1>\epsilon_2>\cdots>0$ be a sequence of elements from $\bR$, such that $\R\cdot \epsilon_i <\epsilon_{i-1}$ for all $i$. Then the linear mapping $p\mapsto\sum_{i=0}^\infty \epsilon_i \cdot p^{(2i)}(0)$ is well-defined and strictly positive in the desired sense. If we further assume  $\R\cdot \epsilon_2< \epsilon_1^2$, then the same argument as in Example \ref{notcp} shows that the Cauchy-Schwarz inequality is not fulfilled.
\end{example}

\begin{definition} A $\C$-linear functional $\varphi\colon A\rightarrow\bC$ with $\varphi(a^*)=\overline{\varphi(a)}$ is called \textit{completely positive}, if for all $m\in\N$, the componentwisely defined function $$\varphi^{(m)}\colon {\rm M}_m(A)\rightarrow {\rm M}_m(\bC)$$ maps sums of hermitian squares to positive semidefinite matrices.
\end{definition}

\begin{remark}\label{boring}
It is easily seen that a positive $\C$-linear functional $\varphi\colon A\rightarrow \C$ with $\varphi(a^*)=\overline{\varphi(a)}$ is always completely positive.
\end{remark}

\begin{example}
The functionals from Example \ref{notcp} and Example \ref{notcp2} are positive, but not completely positive. Indeed, with $a=1+t^2$ and  $M=\left(\begin{array}{cc}1 & a \\0 & 0\end{array}\right)$ we find that $$\varphi^{(2)}\left(M^*M\right)=\left(\begin{array}{cc}1 & \varphi(a) \\\varphi(a^*) & \varphi(a^*a)\end{array}\right)$$ is not positive semidefinite, since its determinant is negative in $\bR$.
\end{example}

\begin{lemma}\label{fext}
A  $\C$-linear functional $\varphi\colon A\rightarrow \bC$ with $\varphi(a^*)=\overline{\varphi(a)}$ is completely positive if and only if the $\bC$-linear extension $$\id\otimes\varphi\colon \bC\otimes_\C A\rightarrow \bC$$ is positive.
\end{lemma}
\begin{proof}
The condition that $\id\otimes\varphi$ is positive is \begin{align*} 0& \leq (\id\otimes\varphi)\left(\left(\sum_{j=1}^m z_j\otimes a_j\right)^*\left(\sum_{j=1}^m z_j\otimes a_j\right)\right)\\ &= (\id\otimes\varphi)\left(\sum_{j,k} \overline{z}_jz_k\otimes a_j^*a_k\right)\\ &= \sum_{j,k} \overline{z}_jz_k \cdot \varphi(a_j^*a_k)\end{align*} for all $m\in\N$, $z_j\in\bC, a_j\in A.$ But this just means that the matrix $( a_j^*a_k)_{j,k}$ is mapped to a positive semidefinite matrix under $\varphi^{(m)}$. Since every sum of hermitian squares in ${\rm M}_m(A)$ is a finite sums of such rank one squares, this proves the claim.
\end{proof}

\begin{corollary}
If $\varphi\colon A\rightarrow \bC$ is completely positive, then it fulfills the Cauchy-Schwarz-inequality $$\vert \varphi(a^*b)\vert^2\leq \varphi(a^*a)\varphi(b^*b)$$ for all $a,b\in A$.
\end{corollary}
\begin{proof}
Either consider the positive and $\bC$-linear extension $\id\otimes\varphi$ to $\bC\otimes_\C A$ and use the standard proof for the inequality. Or apply $\varphi^{(2)}$ to the sum of hermitian squares $$\left(\begin{array}{cc}a & b \\0 & 0\end{array}\right)^*\left(\begin{array}{cc}a & b \\0 & 0\end{array}\right)=\left(\begin{array}{cc}a^*a & a^*b \\b^*a & b^*b\end{array}\right),$$ and use that the obtained matrix is positive semidefinite. \end{proof}

\begin{remark}
We see from the last proof that in fact only the $2$-positivity of $\varphi$ is needed for the Cauchy-Schwarz-inequality.
\end{remark}

\begin{corollary}\label{gauge}
Let $A$ be a $\C$-algebra with involution and $\bR$ a real closed field which contains $\R$. Let $\varphi \colon A \to \bC$ be a completely positive $\C$-linear functional  which satisfies $\varphi(a^*) = \overline{\varphi(a)}$ for all $a\in A$. The gauge $\|a\|_{\varphi}:= \varphi(a^*a)^{1/2}$ satisfies
$$\|\lambda \cdot a\|_{\varphi} = |\lambda| \cdot \|a\|_{\varphi}$$
and 
$$\|a+b\|_{\varphi} \leq \cdot \|a\|_{\varphi} + \|b\|_{\varphi}.$$
\end{corollary}
\begin{proof}
The first assertion is obvious. Let's compute 
\begin{eqnarray*}
\|a+b\|^2_{\varphi} &=& \varphi((a+b)^*(a+b)) \\
&=& \varphi(a^*a) + \varphi(a^*b) + \varphi(b^*a) + \varphi(b^*b) \\
&\leq& \|a\|^2_{\varphi} + \|b\|^2_{\varphi} + 2\vert \varphi(a^*b)\vert \\
&\leq&  \|a\|_{\varphi}^2 + \|b\|_{\varphi}^2 + 2\Vert a\Vert_\varphi \Vert b\Vert_\varphi \\ 
& = & (\Vert a\Vert_\varphi + \Vert b \Vert_\varphi)^2.
\end{eqnarray*}
This proves the claim.
\end{proof}

It turns out that separation from the cone of sums of hermitian squares can often be done with a completely positive functional.

\begin{definition}
Let $A$ be a $\C$-algebra with involution, not necessarily unital. Then $A$ is called {\it real reduced} if $\sum_i a_i^*a_i=0$ implies $a_i=0$ for all $i$ and $a_i\in A$. 
\end{definition}

\begin{theorem}\label{cpsep}  Let $A$ be a $\C$-algebra with involution, that is real reduced. Let $b\in A^h\setminus\Sigma^2 A$. Then there is real closed extension field $\bR$ of $\R$ and a completely positive $\C$-linear functional $\varphi\colon A\rightarrow \bC$ with $\varphi(a^*)=\overline{\varphi(a)}$, such that  $$\varphi(b)<0\quad  \mbox{ and }\quad  \varphi(a^*a)>0 \mbox{ for } a\in A\setminus\{0\}.$$
\end{theorem}
\begin{proof} For any finite dimensional subspace $H$ of $A$ denote by $$\Sigma^2 H=\left\{\sum_i a_i^*a_i\mid a_i\in H\right\} $$ the set of sums of hermitian squares of elements from $H$. It is a well-known fact that $\Sigma^2 H$ is a closed convex cone in a finite dimensional subspace of $A^h$. This follows from the fact that $A$ is real reduced, using for example the approach from \cite{MR1823953}, especially Lemma 2.7 from that work. It also follows that $\Sigma^2 H$ is salient, i.e. fulfills $$\Sigma^2 H\cap -\Sigma^2 H=\{0\}.$$So for each such $H$ there is an $\R$-linear functional $\varphi_H\colon A^h\rightarrow \R$ with $$\varphi(b)<0 \quad \mbox{and}\quad \varphi_H(a^*a)>0 \mbox{ for all } a\in H\setminus\{0\}.$$

Let $\mathcal{S}$ be the set of all finite dimensional subspaces $H$ of $A$, equipped with an ultrafilter $\omega$ containing all the sets $\{H\in\mathcal{S}\mid c\in H\}.$ Define $$\varphi\colon A^h\rightarrow \R^\omega; \quad \varphi(a):=\left(\varphi_H(a)\right)_{H\in\mathcal{S}}.$$ Then $\varphi$ does the separation as desired. We consider the $\C$-linear extension of $\varphi$ to $A$, and finally show that it is completely positive. The $\C$-linear extension of $\varphi_H$ to $A$ indeed maps a matrix $(a_i^*a_j)_{i,j}\in {\rm M}_n(A)$ to a positive semidefinite hermitian matrix, at least if all $a_i\in H$, as is easily checked (compare to Remark \ref{boring}).   Since we can check positivity of the matrix  $\left(\varphi(a_i^*a_j)\right)_{i,j}\in{\rm M}_n(\R^\omega[i])$  componentwisely in ${\rm M}_n(\R[i])$, by the Theorem of \L os, this finishes the proof.  \end{proof}

Throughout, we will take the freedom to consider $*$-representations of $A$ on vector spaces which carry a sesqui-linear $\bC$-valued inner product, where $\bC = \bR[i]$ for some real closed field $\bR \supset \R$. We call these representations \emph{generalized representations}.
For every completely positive functional $\varphi \colon A \to \bC$ we can perform the usual GNS-construction to construct such a representation (see the proof of Theorem \ref{schm} below for more technical details). The usual concepts of self-adjointness and positive semi-definiteness of operators on such a vector space can be defined without any problems.  The first consequence is the following Positivstellensatz (compare to the standard Positivstellensatz from the introduction):

\begin{theorem}\label{pos}
Let $A$ be a real reduced unital $*$-algebra, and   $a \in A^h$. Then $a$ is positive semidefinite in every generalized representation if and only if $a \in \Sigma^2 A$.
\end{theorem}
\begin{proof}
If $a \notin  \Sigma^2 A$, then there exists a completely positive map $\varphi \colon A \to \bC$, such that $\varphi(a)<0$. Clearly, $a$ will not be positive semidefinite in the generalized GNS-representation associated with $\varphi$.
\end{proof}

Examples for real reduced unital $*$-algebras are group algebras $\C[\Gamma]$.
In Section \ref{free}, we will see that for particular groups the study of generalized representations of $\C[\Gamma]$  can be reduced to the study of usual (finite-dimensional) unitary representations, using Tarski's transfer principle.

\section{Sums of Squares in the Group Algebra} \label{aug}

Let $\Gamma$ be a group and let $\C[\Gamma]$ denote the complex group algebra. A typical element in $\C[\Gamma]$ is denoted by $a = \sum_{g} a_g g,$ where only finitely many of the $a_g\in\C$ are not zero. In $\C [\Gamma]$ we identify $\C$ with $\C \cdot e$, where $e$ denotes the neutral element of $\Gamma$.
The group algebra comes equipped with an involution $(\sum_g a_g g)^* := \sum_{g} \bar a_g g^{-1}$ and a trace $\tau \colon \C[\Gamma] \to \C$, which is given by the formula $\tau(\sum_{g} a_g g) = a_e$.  The trace shows that $\C[\Gamma]$ is real reduced. Let $\Sigma^2\C[\Gamma]$ denote the set of sums of hermitian squares in $\C[\Gamma]$. The following appears for example as Example 3 in \cite{MR2494309}:

\begin{lemma} \label{bounded}
$\|a\|^2_1 - a^*a \in \Sigma^2 \C[\Gamma]$ for all $a \in \C[\Gamma]$, where $\Vert a\Vert_1=\sum_g \vert a_g\vert.$  
\end{lemma}
\begin{remark}\label{rem:cim}
From the identity $$2\Vert a\Vert_1\cdot \left( \Vert a\Vert_1-a\right)=\left(\Vert a\Vert_1 -a\right)^*\left(\Vert a\Vert_1 -a\right) + (\Vert a\Vert_1^2 -a^*a)$$ for $a\in\C[\Gamma]^h$ we see  that $1$ is an algebraic interior point of the cone $\Sigma^2\C[\Gamma]$ in the real vector space $\C[\Gamma]^h.$  That means $1+\epsilon a\in\Sigma^2\C[\Gamma]$ for all $a\in\C[\Gamma]^h$ and sufficiently small $\epsilon >0$. In fact, the $\epsilon$ does only depend on $\Vert a\Vert_1$ here.
\end{remark}

\begin{remark}
As explained in the introduction, for any element $a\in\C[\Gamma]^h$ that is positive semidefinite in each (usual) $*$-representation of $\C[\Gamma]$ one thus has  $a+\epsilon \in\Sigma^2\C[\Gamma],$ for all $\epsilon >0$. 
\end{remark}

\begin{remark}
Since $\C[\Gamma]$ is real reduced and unital, the result of Theorem \ref{pos} holds here as well. So if $a$ is positive semidefinite in each generalized representation, then $a\in\Sigma^2 \C[\Gamma].$
\end{remark}

We now consider the augmentation homomorphism $\varepsilon \colon \C[\Gamma] \to \C$ which is defined by $\varepsilon(\sum_{g} a_g g) = \sum_{g} a_g$. 
The \emph{augmentation ideal} is $$\omega(\Gamma) := \ker(\varepsilon) =\left\{a \in \C[\Gamma] \mid \sum_g a_g = 0 \right\}.$$ We set $c(g):=g-1$ and note that $\{c(g) \mid g \in \Gamma \setminus \{e\} \}$ is a basis of $\omega(\Gamma)$. The multiplication satisfies
$$c(g)c(h) = c(gh) - c(g) - c(h).$$ We denote by $\omega^2(\Gamma)$  the square of $\omega(\Gamma),$ i.e. $\omega^2(\Gamma)= {\rm span}_{\C}\{ab \mid a,b \in \omega(\Gamma)\}$. 
Inside $\omega(\Gamma)$, we study the cone of sums of hermitian squares
$$\Sigma^2 \omega(\Gamma):= \left\{ \sum a_i^*a_i \mid a_i \in \omega(\Gamma)\right\}.$$
We are interested in interior points of this cone. Note that $\Sigma^2 \omega(\Gamma) \subset \omega^2(\Gamma)$ and
$\omega(\Gamma)/\omega^2(\Gamma) = \C \otimes_{\Z} \Gamma_{ab}$, where $\Gamma_{ab} = \Gamma/[\Gamma,\Gamma]$ and $[\Gamma,\Gamma]$ denotes the subgroup of $\Gamma$ which is generated by commutators.
Hence, if $\Gamma$ has non-torsion abelianization, then $\Sigma^2\omega(\Gamma)$ is contained in a proper subspace of $\omega(\Gamma)$.
We will however show below that $\Sigma^2\omega(\Gamma)$ always has an interior point in $\omega^2(\Gamma)^h$.

\begin{lemma} For any group, $\Sigma^2 \omega(\Gamma) = \Sigma^2 \C[\Gamma] \cap \omega(\Gamma)$
\end{lemma}
\begin{proof} The inclusion $\Sigma^2 \omega(\Gamma) \subset \Sigma^2 \C[\Gamma] \cap \omega(\Gamma)$ is obvious. If $\sum_i a_i^*a_i \in \omega(\Gamma)$ with $a_i \in \C[\Gamma]$, then $\sum_i |\varepsilon(a_i)|^2 =0$ and hence $\varepsilon(a_i)=0$ for all $i$. This proves the converse inclusion.
\end{proof}

\begin{remark} It turns out that  we can always extend positive functionals $\varphi$ on $\omega(\Gamma)$ to positive functionals $\overline\varphi$ on $\C[\Gamma]$, at least if we allow for an extension of the real closed field.
Indeed observe  $$\left(\Sigma^2\C[\Gamma] + \omega(\Gamma)^h\right)\cap -\left(\Sigma^2\C[\Gamma] + \omega(\Gamma)^h\right)=\omega(\Gamma)^h,$$ which follows immediately from an application of the augmentation homomorphism $\varepsilon.$ We can thus apply Theorem \ref{extend}. \end{remark}

\begin{lemma}\label{lembound} Let $\bR$ be a real closed extension field of $\R$, and $\varphi \colon \omega(\Gamma) \to \bC$ a completely positive $\C$-linear functional with $\varphi(a^*)=\overline{\varphi(a)}$ for all $a\in\omega(\Gamma)$. Then for all $s,h\in\Gamma$
$$|\varphi(c(s)^*c(h))| \leq \frac{1}{\sqrt{2}} \cdot \left( \varphi(c(s)^*c(s)) + \varphi(c(h)^*c(h)) \right).$$
$$\varphi(c(sh)^* c(sh)) \leq  2\cdot  \left( \varphi(c(s)^* c(s)) + \varphi(c(h)^* c(h)) \right).$$
\end{lemma}
\begin{proof}
The first inequality is an application of the Cauchy-Schwarz inequality (that is fulfilled by completely positive functionals) and the inequality $\lambda\mu \leq \frac{(\lambda+\mu)^2}{2}$.  For the second inequality first apply the triangle identity from  Corollary \ref{gauge} to the equation $$c(sh) =c(s) + \left(c(s)c(h)+c(h)\right), $$ together with the well-known inequality $(a + b)^2 \leq 2(a^2 + b^2)$.  Then use the easy to verify identity $\Vert c(s)c(h)+c(h)\Vert^2=\Vert c(h)\Vert^2.$
\end{proof}

 Since $1\notin\omega(\Gamma)$, we need to find a different candidate for an interior point of $\Sigma^2\omega(\Gamma).$
\begin{definition}
Let $S \subset \Gamma$ be a finite symmetric set, i.e. $S^{-1}=S$. We define the Laplace operator on $S$ to be
$$\Delta(S) := |S| - \sum_{s \in S} s$$
\end{definition}

\begin{remark}\label{remlap}
Note that for every finite symmetric set $S \subset \Gamma$
$$\Delta(S) = \frac{1}2 \cdot \sum_{s \in S} c(s)^*c(s) \in \Sigma^2 \omega(\Gamma).$$
\end{remark}

\begin{proposition} \label{boundedbylap}
Let $\Gamma$ be a group, generated by a finite symmetric set $S$.  Then for any $b\in\omega^2(\Gamma)$ there exists a constant $C(b)\in\R$, such that for any real closed extension field $\bR$  of $\R$ and any completely positive $\C$-linear functional 
$\varphi \colon \omega(\Gamma) \to \bC$ with $\varphi(a^*)=\overline{\varphi(a)}$ for all $a\in\omega(\Gamma$), one has  $$|\varphi(b)| \leq C(b) \cdot \varphi(\Delta(S)).$$
\end{proposition}
\begin{proof} Every element  $b\in \omega^2(\Gamma)$ is a finite linear combination of $c(g)^*c(h)$, for $g,h \in \Gamma \setminus \{e\}$. This implies  that $\vert\varphi(b)\vert$ is bounded  by a universal constant times $\max \{ \varphi(c(s)^*c(s)) \mid s \in S \},$ using Lemma \ref{lembound} several times. However, $$\max \{ \varphi(c(s)^*c(s)) \mid s \in S \} \leq 2\cdot \varphi(\Delta(S))$$ follows from Remark \ref{remlap}. This proves the claim.
\end{proof}

\begin{theorem}\label{inner} Let $\Gamma$ be a group with finite generating symmetric set $S$. Then for every $b\in\omega^2(\Gamma)^h$ there is a constant $C(b)\in\R$ such that $$C(b)\cdot \Delta(S) \pm b\in\Sigma^2\omega(\Gamma).$$
In particular, $\Delta(S)$ is an inner point of the cone $\Sigma^2 \omega(\Gamma)$ in $\omega^2(\Gamma)^h$. If $H_1(\Gamma,\C)=0,$ it is an inner point in $\omega(\Gamma)^h$.
\end{theorem}
\begin{proof} In view of Proposition \ref{boundedbylap} we find that $$C(b)\cdot \Delta(S)\pm  b$$ is nonnegative under each completely  positive real-closed valued $\C$-linear functional $\varphi$ on $\omega(\Gamma)$. In view of Theorem \ref{cpsep} this means that $C(b)\cdot\Delta(S)\pm b$ is a sum of hermitian squares in $\omega(\Gamma)$. Note that we use that $\omega(\Gamma)$ is real reduced. Finally note that $H_1(\Gamma,\C)=0$ just means that $\omega^2(\Gamma)=\omega(\Gamma).$
\end{proof}

\section{Groups with Kazhdan's Property (T)}\label{sec:kaz}

We want to show that the constant $C(b)$ from Theorem \ref{inner} can be chosen as a fixed multiple of $\Vert b\Vert_1$, in case the group $\Gamma$ has Kazhdan's property.

\begin{definition} Let $\Gamma$ be a group and $\pi \colon \Gamma \to U(H)$ a unitary representation on a Hilbert space $H$.
\begin{enumerate}
\item A 1-cocycle with respect to the unitary representation $\pi$ is a map $\delta \colon \Gamma \to H$, such that for all $g,h \in \Gamma$, we have $\delta(gh) = \pi(g) \delta(h) + \delta(g)$.
\item A 1-cocycle $\delta \colon \Gamma \to H$ is called inner, if $\delta(g) = \pi(g) \xi - \xi$ for some vector $\xi \in H$.
 \end{enumerate}
\end{definition}

\begin{definition} A group has Kazhdan's property (T) if every 1-cocycle with respect to every unitary representation is inner.
\end{definition}
We will use several results on Kazhdan groups, that can for example be found in \cite{MR2415834}. It is well-known that groups with Kazhdan's property (T) admit a finite generating set $S$, and that
$$\omega^2(\Gamma) = \omega(\Gamma)$$ holds. 
 It is also known that for a fixed finite symmetric and generating set $S$ in a Kazhdan group $\Gamma$ there is some $\epsilon >0$ such that for any unitary representation $\pi\colon\Gamma\rightarrow U(H)$ without fixed vectors one has $$\langle\Delta(S)\xi,\xi\rangle \geq \epsilon\cdot\Vert\xi\Vert^2 \quad \mbox{Êfor all }\xi \in H.$$ Such $\epsilon$ is called a {\it Kazhdan constant} for $S$.

Let's revisit the standard GNS-representation in the context of $\omega(\Gamma)$. Let $\varphi \colon \omega(\Gamma) \to \C$ be a positive linear functional with $\varphi(a^*)=\overline{\varphi(a)}$. We associate to $\varphi$ a Hilbert space as follows. We define on $\omega(\Gamma)$ a positive semi-definite sesqui-linear form
$$\langle a,b \rangle_{\varphi} := \varphi(b^*a)$$
and set $\|a\|_{\varphi} := \langle a,a \rangle^{1/2}_{\varphi}$.
Let $N(\varphi):= \{a \in \omega(\Gamma) \mid \|a\|_{\varphi} =0 \}$ and define $L^2(\omega(\Gamma),\varphi)$ to be the metric completion of $\omega(\Gamma)/N(\varphi)$ with respect to $\|.\|_{\varphi}$. We denote the image of $c(g)$ in $L^2(\omega(\Gamma),\varphi)$ by $\delta(g)$ and denote by $\delta(\Gamma)$ their complex linear span, which is dense by definition of $L^2(\omega(\Gamma),\varphi))$. It is a standard fact the the left-multiplication of $\omega(\Gamma)$ on itself extends to a homomorphism $\pi^{\varphi} \colon \omega(\Gamma) \to D(\delta(\Gamma))$, where $D(\delta(\Gamma))$ denotes the algebra of densely defined linear operators mapping $\delta(\Gamma)$ into itself. Indeed, if $a \in N(\varphi)$ and $b \in \omega(\Gamma)$, then $ba \in N(\varphi)$ since 
$$\varphi((ba)^*ba) = \varphi(a^*b^*ba) \leq \|b\|^2_1 \cdot \varphi(a^*a)$$
by Lemma \ref{bounded}. Note that
$$\pi^{\varphi}(c(g)) \delta(h) = \delta(gh) - \delta(g) - \delta(h).$$

Now we define a unitary representation $\pi_\varphi$  of $\Gamma$ on $L^2(\omega(\Gamma),\varphi)$ by the rule $$\pi_{\varphi}(g) := \pi^{\varphi}(c(g)) + 1_{L^2(\omega(\Gamma),\varphi)}.$$

\begin{lemma} \label{nofixed} If $\omega^2(\Gamma)= \omega(\Gamma)$, then the representation $\pi_{\varphi}$ has no fixed vectors.
\end{lemma}
\begin{proof}
Assume $\eta \in L^2(\omega(\Gamma),\varphi)$ is a fixed vector.  By definition of $\pi_{\varphi}$, this means $\pi^{\varphi}(c(g)) \eta =0$, for all $g \in \Gamma$. Hence, 
$$0 = \langle \pi^{\varphi}(c(g^{-1})) \eta, \delta(h) \rangle_\varphi = \langle \eta, \delta(gh) - \delta(g) - \delta(h) \rangle_\varphi.$$ 
Since $c(g)c(h) = c(gh) - c(g) - c(h)$, the vectors $\delta(gh) - \delta(g) - \delta(h)$ span the image of $\omega^2(\Gamma)$ in $\delta(\Gamma)$ and hence $\delta(\Gamma)$ since $\omega^2(\Gamma)= \omega(\Gamma)$. \end{proof}
 
Note that the map $g \mapsto \delta(g)$ satisfies
$$\delta(gh) = \pi_{\varphi}(g) \delta(h) + \delta(g)$$
and hence defines a 1-cocycle with respect to the representation $\pi_{\varphi}$.
If $\Gamma$ is a Kazhdan group, then there exists $\Omega \in L^2(\omega(\Gamma),\varphi)$ such that
$$\delta(g) = \pi_{\varphi}(g) \Omega - \Omega.$$ 

\begin{proposition}\label{kazbound}
Let $\Gamma$ be a Kazhdan group with finite symmetric generating set $S$ and Kazhdan constant $\epsilon >0.$ Then for every $b\in\omega(\Gamma)^h,$ every real closed extension field $\bR$ of $\R$ and every  positive nontrivial  $\C$-linear functional $\varphi\colon\omega(\Gamma)\rightarrow\bC$ with $\varphi(a^*)=\overline{\varphi(a)}$ one  as $$\epsilon\cdot \varphi(b) < 2\Vert b\Vert_1 \cdot\varphi(\Delta(S)).$$
\end{proposition}
\begin{proof}
Let us first assume $\bR=\R$ and $\bC=\C.$ We do the GNS construction as just described, and get some $\Omega\in L^2(\omega(\Gamma),\varphi)$ with $\delta(g)=\pi_\varphi(g)\Omega-\Omega,$ for all $g\in\Gamma.$ We set $$\overline\varphi\colon\C[\Gamma]\rightarrow\C; \quad \overline{\varphi}(a)=\langle \pi_\varphi(a)\Omega,\Omega\rangle$$ and compute:
\begin{eqnarray*}
\bar \varphi(c(h)^*c(g))&=& \langle \pi_{\varphi}(c(g)) \Omega,\pi_{\varphi}(c(h)) \Omega \rangle_{\varphi} \\
&=& \langle \delta(g), \delta(h) \rangle_{\varphi} \\
&=& \varphi(c(h)^*c(g)).
\end{eqnarray*} This shows that $\bar \varphi$ and $\varphi$ agree on $\omega^2(\Gamma)$ and hence on $\omega(\Gamma)$. If we now do the standard GNS-construction with respect to $\overline\varphi$, we see that there is a natural $\Gamma$-equivariant identification of $L^2(\C[\Gamma],\bar \varphi)$ and $L^2(\omega(\Gamma),\varphi)$.  Since the representation $\pi_\varphi$ has no fixed vectors we get $$\varphi(\Delta(S))=\overline{\varphi}(\Delta(S))= \langle \Delta(S)1,1\rangle_{\overline{\varphi}}\geq\epsilon\cdot \overline{\varphi}(1).$$ Since $\overline\varphi$ is positive and nontrivial, it follows from Remark \ref{rem:cim} that $\overline\varphi(1)>0$. So finally, again using Remark \ref{rem:cim} we find $$\epsilon\cdot\varphi(b)=\epsilon\cdot\overline{\varphi}(b)\leq \epsilon\cdot \Vert b\Vert_1\cdot \overline{\varphi}(1)<2\Vert b\Vert_1\cdot\varphi(\Delta(S)),$$ the desired result.

Now let $\bR$ be arbitrary and $\varphi\colon\omega(\Gamma)\rightarrow \bC$ positive and nontrivial. From Theorem \ref{inner} it follows that $\varphi(\Delta(S))>0.$ So we can assume without loss of generality that $\varphi(\Delta(S))=1$. Again from Theorem \ref{inner} we see that $\varphi$ now only takes values in $\mathcal{O}[i]$, where $\mathcal{O}$ is the convex hull of $\R$ in $\bR$. It is well known that $\mathcal{O}$ is a valuation ring in $\bR$ with maximal ideal $\mathfrak{m}$, and that $\mathcal{O}/\mathfrak{m}=\R$. The residue map $\mathcal{O}\rightarrow\mathcal{O}/\mathfrak{m}$ maps nonnegative elements to nonnegative elements.  So if we compose $\varphi$ with the residue map on $\mathcal{O}[i]$, we get a positive linear functional to $\C$. Since we know that the desired strict inequality holds now, it was already valid for $\varphi.$
\end{proof}

\begin{theorem}\label{innerkaz} 
Let $\Gamma$ be a group with finite generating symmetric set $S$. Consider the statements:
\begin{enumerate}
\item $\Gamma$ has Kazhdan's property (T).
\item $\Delta(S)$ is an algebraic interior point of $\Sigma^2\omega(\Gamma)$ in the $\ell^1$-metric of $\omega(\Gamma)^h$. More precisely, there exists a constant $\epsilon > 0$, such that for every $b\in\omega(\Gamma)^h$ with $\|b\|_1=1$ we have $$\Delta(S) + \epsilon \cdot b\in\Sigma^2\omega(\Gamma).$$
\end{enumerate}
The following implications hold: $(1)$ implies $(2)$ and $(2)$ implies $(1)$ under the additional assumption $H_2(\Gamma,\C)=0$.
\end{theorem}
\begin{proof}
Implication $(1) \Rightarrow (2)$ is a direct consequence Theorem \ref{cpsep} and Proposition \ref{kazbound}. 
Let us now prove $(2) \Rightarrow (1)$  under the additional assumption $H_2(\Gamma,\C)=0$. We first prove two lemmas.

\begin{lemma} Let $\Gamma$ be a group.
There is an exact sequence as follows:
$$0 \to H_2(\Gamma,\C) \to \omega(\Gamma) \otimes_{\C[\Gamma]} \omega(\Gamma) \to \omega(\Gamma) \to H_1(\Gamma,\C) \to 0$$
\end{lemma}
\begin{proof}
It is well-known that $$\omega(\Gamma)/\omega(\Gamma)^2 = \Gamma_{ab} \otimes_{\Z} \C = H_1(\Gamma,\C).$$ This shows exactness at $\omega(\Gamma)$ and it remains to show exactness at $\omega(\Gamma) \otimes_{\C[\Gamma]} \omega(\Gamma)$.
Since $0 \to \omega(\Gamma) \to \C[\Gamma] \to \C \to 0$ is exact, we have
$$H_2(\Gamma,\C) = H_1(\Gamma,\omega(\Gamma)) = {\rm Tor}_1^{\C[\Gamma]}(\C,\omega(\Gamma)).$$ Again, we get an exact sequence
$${\rm Tor}_1^{\C[\Gamma]}(\C[\Gamma],\omega(\Gamma)) \to {\rm Tor}_1^{\C[\Gamma]}(\C,\omega(\Gamma)) \to \omega(\Gamma) \otimes_{\C[\Gamma]} \omega(\Gamma) \to \C[\Gamma] \otimes_{\C[\Gamma]} \omega(\Gamma).$$
This finishes the proof since ${\rm Tor}_1^{\C[\Gamma]}(\C[\Gamma],\omega(\Gamma)) = 0.$
\end{proof}

If $\Delta(S)$ is an algebraic interior point of $\Sigma^2 \omega(\Gamma)$ in $\omega(\Gamma)^h$, then $\omega(\Gamma)=\omega^2(\Gamma)$, i.e.  $H_1(\Gamma,\C) =0$. Hence $H_2(\Gamma,\C)=0$ ensures that the natural map $$\omega(\Gamma) \otimes_{\C[\Gamma]} \omega(\Gamma) \to \omega(\Gamma)$$ is an isomorphism. This is what we are going to use.

\begin{lemma} \label{lem1}
Let $\pi \colon \Gamma \to U(H)$ be a unitary representation and $\delta \colon \Gamma \to H$ be a 1-cocycle with respect to $H$. Then,
$$\varphi(c(h)^* \otimes c(g)) := \langle \delta(g),\delta(h)\rangle$$
yields a well-defined positive linear functional on $\omega(\Gamma)=\omega(\Gamma) \otimes_{\C[\Gamma]} \omega(\Gamma)$.

\end{lemma}\label{bla}
\begin{proof}
It is clear that $(c(h)^*,c(g)) \mapsto \langle \delta(g),\delta(h)\rangle$ defines a bilinear map on $\omega(\Gamma)$, i.e. a linear map 
$\varphi' \colon \omega(\Gamma) \otimes_{\C} \omega(\Gamma) \to \C$. We show that this map passes to $\omega(\Gamma) \otimes_{\C[\Gamma]} \omega(\Gamma)$. Let $g,h,k \in \Gamma$, then
\begin{eqnarray*}
\varphi'(c(h)^*k \otimes c(g)) &=& \varphi'(c(h^{-1})k \otimes c(g))\\
&=&\varphi'((c(h^{-1}k) - c(k) ) \otimes c(g)) \\
&=&\varphi'(c(k^{-1}h)^* \otimes c(g)) - \varphi'(c(k^{-1})^* \otimes c(g) ) \\
&=& \langle \delta(g), \delta(k^{-1}h) \rangle - \langle \delta(g) , \delta(k^{-1}) \rangle \\
&=& \langle \delta(g), \pi(k^{-1}) \delta(h) \rangle \\
&=& \langle \pi(k) \delta(g),  \delta(h) \rangle \\
&=& \langle \delta(kg) ,\delta(h) \rangle - \langle \delta(k) , \delta(h) \rangle \\
&=& \varphi'(c(h)^* \otimes c(kg)) - \varphi'(c(h)^* \otimes c(k) ) \\
&=& \varphi'(c(h^*) \otimes k c(g)).
\end{eqnarray*} We can now understand $\varphi'$ as a linear map on $\omega(\Gamma),$ via the above isomorphism to  $\omega(\Gamma)\otimes_{\C[\Gamma]}\omega(\Gamma).$ Since $a^*a$ corresponds to $a^*\otimes a$ one easily checks that $\varphi$ is positive on $\omega(\Gamma)$.
\end{proof}

We continue with the proof of Theorem \ref{innerkaz}. Condition (2) in Theorem \ref{innerkaz} and Lemma \ref{bla} imply that any 1-cocycle with respect to any unitary representation is bounded. This is well-known to imply Kazhdan's property (T) for $\Gamma$.
\end{proof}

\begin{remark}
It is not clear whether the condition $H_2(\Gamma,\C)=0$ is necessary.
\end{remark}

There is an analogue of the implication $(1) \Rightarrow (2)$ in Theorem \ref{innerkaz} in $\ell^1 \Gamma := \{ \sum_{g \in \Gamma} a_g g \mid \sum_{g \in \Gamma} |a_g|< \infty \}$. We set $\omega^{1} \Gamma := \overline{\omega(\Gamma)}^{\|.\|_1}$ and define
$$\Sigma^{2,1} \omega(\Gamma):= \left\{ \sum_{i=1}^{\infty} a_i^*a_i \mid a_i \in \omega[\Gamma], \sum_{i=1}^{\infty} \|a_i\|^2_1< \infty \right\}$$ and
$$\Sigma^{2,1} \ell^1(\Gamma):= \left\{ \sum_{i=1}^{\infty} a_i^*a_i \mid a_i \in \ell^1[\Gamma], \sum_{i=1}^{\infty} \|a_i\|^2_1< \infty \right\}.$$

We note that $\|a\|_1 - a \in \Sigma^{2,1} \ell^1(\Gamma)$ for every hermitean element $a \in \ell^1 \Gamma$. Indeed,
\begin{eqnarray*}
\|a\|_1 - a &=& \sum_{g \in G} 2|a_g| - a_g g - \bar a_g g^{-1} \\
&=& \sum_{g \in G} \left(|a_g|^{1/2} - \frac{a_g}{|a_g|^{1/2}} g \right)^* \left(|a_g|^{1/2} - \frac{a_g}{|a_g|^{1/2}} g \right).
\end{eqnarray*}
Hence, for $\varphi \colon \ell^1 \Gamma \to \bC$, $\C$-linear and positive on $\Sigma^{2,1} \ell^1 \Gamma$, $|\varphi(a)| \leq \|a\|_1$ for all $a \in \ell^1 \Gamma$.
A priori, there is no reason to assume that $\Sigma^{2,1} \ell^1 \Gamma$ or $\Sigma^{2,1} \omega^1 \Gamma$ are closed or have non-trivial interiour. Nevertheless, our result shows:

\begin{corollary}\label{innerkaz2} Let $\Gamma$ be a Kazhdan group with finite generating symmetric set $S$ and Kazhdan constant $\epsilon$. Then for every $b\in\omega^1(\Gamma)^h$ with $\|b\|_1=1$ we have $$ \Delta(S) + \varepsilon \cdot b\in\Sigma^{2,1}\omega(\Gamma).$$
\end{corollary}

\section{Group Algebras of Free Groups} \label{free}

During this section let $\Gamma=F_n$ be the free group on $n$ generators $g_1,\ldots,g_n$ or $\Gamma= F_{\infty}$. Schm\"udgen \cite{schm} has proven that an element from the group algebra $\C[\Gamma]$ that is nonnegative under each finite dimensional $*$-representation is a sum of squares. We demonstrate how his result  can be re-proven with our real closed separation approach. The main idea of our proof is the same as in Schm\"udgen's work. However, instead of a partial GNS-construction we use a full GNS-construction, but over a general real closed field. We then reduce to the standard real numbers by Tarski's transfer principle.

\begin{theorem}[Schm\"udgen] \label{schm}Let $\Gamma=F_n$ be the free group on $n$ generators.
If $b\in\C[\Gamma]^h$ is mapped to a positive semidefinite matrix under each finite dimensional $*$-representation  of $\C[\Gamma]$, then $b\in\Sigma^2\C[\Gamma].$ 
\end{theorem}
\begin{proof} Assume that $b\notin\Sigma^2\C[\Gamma].$ By Theorem \ref{cpsep} there is a real closed extension field $\bR$ of $\R$ and a completely positive $\C$-linear functional $\varphi\colon\C[\Gamma]\rightarrow\bC$ with $\varphi(a^*)=\overline{\varphi(a)}$, such that  $\varphi(b)<0.$ By Lemma \ref{fext}, the canonical $\bC$-linear extension   of $ \varphi$  to $A=\bC\otimes_\C \C[\Gamma]$ is still positive, and we denote it again by $\varphi$. We apply the usual GNS-construction to $A$. We note that  $$N=\{ a\in A \mid \varphi(a^*a)=0\}$$  is a $*$-subspace of the  $\bC$-vector space $A$, which follows from the Cauchy-Schwarz inequality, as shown in Corollary \ref{gauge}. We denote the quotient space $A/N$ by $H$, and note that $$\langle a+N,c+N\rangle_\varphi :=\varphi(c^*a)$$ is a well-defined and positive definite $\bC$-valued sesqui-linear form on $H$. We also note that left-multiplication from $\C[\Gamma]$ on $A$ is well-defined on $H$, as explained in Section \ref{sec:kaz}. So we have a $\C$-linear $*$-representation $$\pi\colon\C[\Gamma]\rightarrow \mathcal{L}(H)$$  with   $\langle\pi(b)\xi,\xi\rangle_\varphi<0,$ where $\xi=1+N$.

Now let $H'$ be a finite dimensional $*$-subspace of $H$, containing the residue classes of all words in the generators $g_i$  of length at most $d$, where $d$ is the maximal word length in $b$. We can choose an orthonormal basis $v_1,\ldots v_m$ of $H'$, using the usual Gram-Schmidt procedure over $\bC$. So there is an orthogonal projection map $p\colon H\rightarrow H',$ defined as $$p\colon h\mapsto \sum_{i=1}^m \langle h,v_i\rangle v_i.$$ Define $$M_i:=p\circ \pi(g_i)\in\mathcal{L}(H').$$ It is easy to see that all $M_i$ are contractions, and thus the linear operators $\sqrt{1-M_i^*M_i}$ and $\sqrt{1-M_iM_i^*}$  exist on $H'$. Using Choi's matrix trick \cite[Theorem 7]{choi}, we define $$U_i:= \left(\begin{array}{cc}M_i & \sqrt{I-M_iM_i^*} \\ \sqrt{I-M_i^*M_i} & -M_i^*\end{array}\right)\in\mathcal{L}(H'\oplus H').$$ The $U_i$ are checked to be unitary operators, and thus yield a $\C$-linear  $*$-representation $\widetilde{\pi}$ of $\C[\Gamma]$ on $H'\oplus H'$. Since the residue classes of all words occuring in $b$ belong to $H',$ and by the definition of the $U_i$, we find $$\langle \widetilde\pi(b)\xi',\xi'\rangle_{H'\oplus H'} =\langle \pi(b)\xi,\xi\rangle_\varphi<0,$$ where $\xi'=(\xi,0).$
Now finally, since $H'\oplus H'$ is finite dimensional, the existence of such a representation over $\bC$ implies the existence over $\C$, by Tarski's transfer principle. This finishes the proof.
\end{proof}

\begin{remark}
The proof becomes even simpler when considering the $*$-algebra of polynomials in non-commuting variables $\C\langle y_1,\ldots,y_n, z_1,\ldots,z_n \rangle$ with $y_i^*=z_i$, or $\C\langle z_1,\ldots,z_n\rangle$ with $z_i^*=z_i,$ instead of the group algebra of a free group. The reason is that one is not  forced to make the matrices $M_i$ unitary (only hermitian in the second case). So Theorem \ref{schm} also  holds for these polynomial algebras. This was first proven by Helton \cite{MR1933721}.
\end{remark}

It is an interesting problem to study the class of groups $\Gamma$, for which positivity of $a \in \C[\Gamma]$ in every finite-dimensional unitary representation implies that $a \in \Sigma^2 \C[\Gamma]$. 
It is clear that in order for an analogous argument to work, $\Gamma$ has to be residually finite-dimensional in a very strong sense. Residual finite-dimensionality of a group means that every unitary representation on a Hilbert space can be approximated in the Fell topology by finite-dimensional representations, see \cite{boz} for details. If -- more generally -- every generalized  unitary representation of $\Gamma$ on a Hilbert space can be approximated on finitely many vectors by generalized finite-dimensional unitary representations, then everything works. With additional work, this can be carried out for virtually free groups \cite{schm}. 

Deep results of Scheiderer \cite{scheid} imply that the conclusion holds for $\Z^2$. However, by a classical results of Rudin \cite{rudin}, the group $\Z^3$ does not satisfy the desired conclusion, and the same holds for every group which contains $\Z^3$. This is also implied by seminal work of Scheiderer, who showed that \cite[Theorem 6.2]{schei1999} the existence of positive elements which are not sum of squares under general assumptions in dimension $\geq 3$.

This shows that the theory of generalized unitary representations is fundamentally different and new pathologies occur.

An intruiging and possibly manageable case is the case of surface groups. Lubotzky and Shalom showed that surface groups are residually finite-dimensional \cite{lubsh}. It is quite possible that their methods extend and lead to a resolution of the case of surface groups.

\begin{conjecture}
Let $G$ be a surface group. Every element $a \in \C[\Gamma]^h$, which is positive semidefinite in every finite-dimensional unitary representation lies in $\Sigma^2 \C[\Gamma]$.
\end{conjecture}

Similar questions can be studied if one allows the unitary representations to be infinite dimensional. Again, the only known obstruction is $\Z^3 \subset \Gamma$.


\begin{bibdiv}
\begin{biblist}

\bib{barvinok}{book}{
  AUTHOR = {A. Barvinok},
   TITLE = {A course in convexity},
  SERIES = {Graduate Studies in Mathematics},
  VOLUME = {54},
PUBLISHER = {American Mathematical Society},
 ADDRESS = {Providence, RI},
    YEAR = {2002},
   PAGES = {x+366},
    ISBN = {0-8218-2968-8},
 MRCLASS = {52-02 (49N15 52-01 90-02 90C05 90C22 90C25)},
}



\bib{MR2415834}{book}{
   author={Bekka, B.},
   author={de la Harpe, P.},
   author={Valette, A.},
   title={Kazhdan's property (T)},
   series={New Mathematical Monographs},
   volume={11},
   publisher={Cambridge University Press},
   place={Cambridge},
   date={2008},
   pages={xiv+472},
}


\bib{boz}{book}{
   author={Brown, N. P.},
   author={Ozawa, N.},
   title={$C^{\ast}$-algebras and finite-dimensional approximations.},
   series={Graduate Studies in Mathematics},
   volume={88},
   publisher={Amer. Math. Soc.},
   place={Providence, RI},
   date={2008},
}



\bib{choi}{article}{
   author={Choi, M. D.},
   title={The full $C^{\ast}$-algebra of the free group on two generators},
   journal={Pacific J. Math},
   volume={87},
   date={1980},
   number={1},
   pages={41--48},
}



\bib{MR2494309}{article}{
    AUTHOR = {Cimpri{\v{c}}, J.},
     TITLE = {A representation theorem for {A}rchimedean quadratic modules
              on {$*$}-rings},
   JOURNAL = {Canad. Math. Bull.},
  FJOURNAL = {Canadian Mathematical Bulletin. Bulletin Canadien de
              Math\'ematiques},
    VOLUME = {52},
      YEAR = {2009},
    NUMBER = {1},
     PAGES = {39--52},
}



\bib{CMN}{article}{
   author={Cimpri{\v{c}}, J.},
   author={Netzer, T.},
   author={Marschall, M.},
   title={Closures of Quadratic Modules},
   journal={Israel Journal of Mathematics},
   volume={183},
   date={2011},
   number={1},
   pages={445--474},
}






\bib{MR1933721}{article}{
   author={Helton, J.W.},
   title={Positive noncommutative polynomials are sums of squares},
   journal={Ann. of Math. (2)},
   volume={156},
   date={2002},
   number={2},
   pages={675--694},
}

\bib{lubsh}{article}{
   author={Lubotzky, A.},
   author={Shalom, Y.},
   title={Finite representations in the unitary dual and Ramanujan groups},
   journal={Discrete geometric analysis: proceedings of the first JAMS Symposium on Discrete Geometric Analysis, December 12-20, 2002, Sendai, Japan, Contemporary Mathematics},
   volume={347},
   date={2004},
   pages={pp. 173},
}

\bib{rudin}{article}{
   author={Rudin, W.},
   title={The extension problem for positive-definite functions},
   journal={Illinois J. Math.},
   volume={7},
   date={1963},
   pages={532--539},
 }

\bib{MR1823953}{article}{
    AUTHOR = {Powers, V.},
    author ={Scheiderer, C.}
     TITLE = {The moment problem for non-compact semialgebraic sets},
   JOURNAL = {Adv. Geom.},
  FJOURNAL = {Advances in Geometry},
    VOLUME = {1},
      YEAR = {2001},
    NUMBER = {1},
     PAGES = {71--88},
}


\bib{pd}{book}{
    AUTHOR = {A. Prestel and C.N. Delzell},
     TITLE = {Positive polynomials},
    SERIES = {Springer Monographs in Mathematics},
      NOTE = {From Hilbert's 17th problem to real algebra},
 PUBLISHER = {Springer-Verlag},
   ADDRESS = {Berlin},
      YEAR = {2001},
}

\bib{schei1999}{article}{
   author={Scheiderer, C.},
   title={Sums of squares of  regular functions on real algebraic varieties},
   journal={Trans. Am. Math. Soc.},
   volume={352},
   date={1999},
   pages={1039--1069},
}


\bib{scheid}{article}{
   author={Scheiderer, C.},
   title={Sums of squares on real algebraic surfaces},
   journal={Manuscr. math},
   volume={119},
   date={2006},
   pages={395--410},
}

\bib{schm2}{article}{
   author={Schm{\"u}dgen, K.},
   title={The $K$-moment problem for compact semi-algebraic sets},
   journal={Math. Ann.},
   volume={289},
   date={1991},
   number={2},
   pages={203--206},
 }

\bib{MR2500470}{article}{
   author={Schm{\"u}dgen, K.},
   title={Noncommutative real algebraic geometry---some basic concepts and
   first ideas},
   conference={
      title={Emerging applications of algebraic geometry},
   },
   book={
      series={IMA Vol. Math. Appl.},
      volume={149},
      publisher={Springer},
      place={New York},
   },
   date={2009},
   pages={325--350},
}

\bib{schm}{article}{
   author={Schm{\"u}dgen, K.},
   title={private communication},
}



\end{biblist}
\end{bibdiv}
\end{document}